\documentclass[12pt]{amsart}
\usepackage{amsmath}
\usepackage{amsthm}
\usepackage{amssymb}
\usepackage{amscd}
\usepackage{hyperref}
\usepackage{tikz}
\usetikzlibrary{matrix,arrows}
\usepackage{mabliautoref}

\usepackage{amsmath,amscd,amssymb,amsopn,amsmath,amsthm,mathrsfs,graphics,amsfonts,enumerate,verbatim,calc}
\usepackage[all]{xy}

\newcommand{\Hom}{\operatorname{Hom}}

\newcommand{\Spec}{\operatorname{Spec}}
\newcommand{\length}{\ell}

\newcommand{\m}{\frak{m}}
\DeclareMathOperator{\pa}{\text{par}}
\DeclareMathOperator{\fp}{\mathfrak{p}}
\DeclareMathOperator{\fq}{\mathfrak{q}}
\DeclareMathOperator{\n}{\mathfrak{n}}

\DeclareMathOperator{\Ann}{Ann}

\numberwithin{equation}{section}

\begin{document}

\title[A Buchsbaum theory for tight closure]{A Buchsbaum theory for tight closure}

\author[Linquan Ma]{Linquan Ma}
\address{Department of Mathematics, Purdue University, West Lafayette, IN 47907, USA}
\email{ma326@purdue.edu}

\author[Pham Hung Quy]{Pham Hung Quy}
\address{Department of Mathematics, FPT University, Hanoi, VIETNAM}
\email{quyph@fe.edu.vn}

\thanks
{LM is partially supported by NSF Grant DMS \#1901672, NSF FRG Grant \#1952366, and a fellowship from the Sloan Foundation. PHQ is partially supported by Vietnam Academy of Science and Technology under grant number CNXS02.01/22-23. This article has been written during many visits of the second author to Vietnam Institute for Advanced Studies in Mathematics. He thanks sincerely the institute for their hospitality and valuable supports. The authors are grateful to Sarasij Maitra, Thomas Polstra, Jugal Verma, and the referee for several useful comments.}


\maketitle

\begin{abstract}
A Noetherian local ring $(R,\m)$ is called Buchsbaum if the difference $\length(R/\fq)-e(\fq, R)$, where $\fq$ is an ideal generated by a system of parameters, is a constant independent of $\fq$. In this article, we study the tight closure analog of this condition. We prove that in an unmixed excellent local ring $(R,\m)$ of prime characteristic $p>0$ and dimension at least one, the difference $e(\fq, R)-\length(R/\fq^*)$ is independent of $\fq$ if and only if the parameter test ideal $\tau_{\pa}(R)$ contains $\m$. We also provide a characterization of this condition via derived category which is analogous to Schenzel's criterion for Buchsbaum rings. 
\end{abstract}


\section{Introduction}

Recall that in a Noetherian local ring $(R,\m)$, if $\fq\subseteq R$ is an ideal generated by a system of parameters, then the length $\length(R/\fq)$ is always greater than or equal to the Hilbert-Samuel multiplicity $e(\fq, R)$. Moreover, $R$ is Cohen-Macaulay if and only if $\length(R/\fq)=e(\fq, R)$ for one (or equivalently, all) such $\fq$. In general, the difference $\length(R/\fq) -e(\fq, R)$ encodes interesting homological properties of the ring $R$. For instance, it is known that under mild assumptions, $\length(R/\fq) -e(\fq, R)$ is uniformly bounded above for all parameter ideals $\fq\subseteq R$ if and only if $R$ is Cohen-Macaulay on the punctured spectrum. A more interesting and subtle condition is that the difference $\length(R/\fq) -e(\fq, R)$ does not depend on $\fq$. Rings that satisfy this latter condition are called {\it Buchsbaum}, and they have been studied extensively, see \cite{StuckradVogel73,StuckradVogelTowardsBuchsbaumSingularities,StuckradVogelBuchsbaumRingsBook,SchenzelApplicationsofDualizingComplex,GotoOntheAssociatedGradedRingsofBuchsbaumRings,GotoYamagishi}, among many others. We briefly summarize the classical theory into the following theorem (see Section 2 and 3 for unexplained terminology):

\begin{theorem}
\label{thm: Buchsbaum theory}
Let $(R,\m, k)$ be a Noetherian local ring of dimension $d$. Let $\fq$ denote an ideal generated by a system of parameters. Then the following conditions are equivalent:
\begin{enumerate}[(1)]
  \item The difference $\length(R/\fq)-e(\fq, R)$ is independent of $\fq$ (i.e., $R$ is Buchsbaum).
  \item $\length(\fq^{\lim}/\fq)$ is independent of $\fq$.\footnote{This characterization of Buchsbaum ring is well-known to experts but we cannot find an explicit reference in the literature, thus we include a short explanation in Remark \ref{rmk: limit closure Buchsbaum}. }
  \item For every system of parameters $x_1,\dots,x_d$, we have $\m\cdot ((x_1,\dots,x_i):x_{i+1})\subseteq (x_1,\dots,x_i)$ for every $i$.
  \item The truncation $\tau^{<d}\mathbf{R}\Gamma_\m R$ is quasi-isomorphic to a complex of $k$-vector spaces.
\end{enumerate}
\end{theorem}

Now we suppose $(R,\m)$ is a Noetherian local ring of prime characteristic $p>0$. A classical result of Kunz \cite{Kunz1} shows that $R$ is regular if and only if the Frobenius endomorphism $F$: $R\to R$, $x\mapsto x^p$, is flat. Kunz's theorem is the starting point in the study of singularities via the Frobenius map. With the development of tight closure theory in Hochster--Huneke's seminal works \cite{HochsterHuneke1,HochsterHuneke2,HochsterHuneke3}, there is an explosion in the understanding of these ``F-singularities". Among them, we recall that a Noetherian local ring $(R,\m)$ is {\it F-rational} if $\fq^*=\fq$ for all ideals $\fq\subseteq R$ that is generated by a system of parameters, where $\fq^*$ denotes the tight closure of $\fq$ (see Section 2 for detailed definitions). Excellent F-rational rings are Cohen-Macaulay: this is basically due to the fact that tight closure captures certain colon ideals of parameter ideals. In connection with Hilbert-Samuel multiplicity, Goto--Nakamura \cite{GotoNakamuraMultiplicityTightClosureParameters,GotoNakamuraBoundofDifference} have shown that, under mild assumptions, $R$ is F-rational if and only if $e(\fq, R)=\length(R/\fq^*)$ for one (or equivalently, all) $\fq$, and $R$ is F-rational on the punctured spectrum if and only if the difference $e(\fq, R)- \length(R/\fq^*)$ is uniformly bounded above for all parameter ideals $\fq$.

Roughly speaking, we can think of F-rationality as a tight closure analog and strengthening of the Cohen-Macaulay property. From this perspective, the aforementioned results of Goto--Nakamura strongly suggest that there should be a tight closure analog of Buchsbaum theory. More precisely, it seems very natural to ask what rings satisfy the property that the difference $e(\fq, R)-\length(R/\fq^*)$, or $\length(\fq^*/\fq)$, does not depend on the parameter ideal $\fq$. This question has been studied by the second author in \cite{PhamHungQuyTightClosureParameterIdeal} and partial results are obtained. In this paper, we develop such a Buchsbaum theory for tight closure and completely answer some questions raised in \cite{PhamHungQuyTightClosureParameterIdeal}. We prove the following result which is in parallel with Theorem \ref{thm: Buchsbaum theory} (again, we refer to Section 2 for unexplained terms below):

\begin{mainthm*}[=Theorem \ref{thm: main theorem}]
Let $(R,\m, k)$ be an unmixed excellent local ring of prime characteristic $p>0$ and dimension $d\geq 1$. Let $\fq$ denote an ideal generated by a system of parameters. Then the following conditions are equivalent:
\begin{enumerate}[(1)]
  \item The difference $e(\fq, R)-\length(R/\fq^*)$ is independent of $\fq$.
  \item $\length(\fq^*/\fq)$ is independent of $\fq$.
  \item $\m\fq^*\subseteq \fq$ for every $\fq$, that is, $\tau_{\pa}(R)$ contains $\m$.
  \item The ${}^*$-truncation $\tau^{<d,*}\mathbf{R}\Gamma_\m R$ is quasi-isomorphic to a complex of $k$-vector spaces.
\end{enumerate}
\end{mainthm*}

This paper is organized as follows. In Section 2 we collect preliminary results on tight closure, local cohomology, Buchsbaum and generalized Cohen-Macaulay modules, and basic knowledge of derived category that will be used throughout. In Section 3 we introduce limit closure and prove $(1)\Leftrightarrow(2)$ of the Main Theorem. Section 4 is devoted to an affirmative answer to a question proposed by the second author in \cite{PhamHungQuyTightClosureParameterIdeal} about system of parameters that are contained in the parameter test ideal, which will imply $(3)\Rightarrow(2)$ of the Main Theorem. Finally, in Section 5 we conclude by proving the remaining parts of the Main Theorem.

\subsection*{Notations and Conventions} All rings appear in this paper are commutative with multiplicative identity $1$. We will often use $(R,\m,k)$ to denote a Noetherian local ring with unique maximal ideal $\m$ and residue field $k=R/\m$. We refer the reader to \cite[Chapter 1-4]{BrunsHerzogBook} for some basic notions such as Cohen-Macaulay rings, regular sequence, Koszul complex, and the Hilbert-Samuel multiplicity. We refer the reader to \cite[Chapter 13]{MatsumuraCommutativeAlgebra} for definition and basic properties of excellent rings.

\section{Preliminaries}
\subsection*{Tight closure, test element, and test exponent} Let $R$ be a Noetherian ring of prime characteristic $p>0$ and let $I\subseteq R$ be an ideal of $R$. Set $R^{\circ} := R \setminus \bigcup_{\fp \in \mathrm{Min}R} \fp$. We recall that the {\it tight closure} of $I$ is the ideal
$$I^* := \{x \mid cx^{p^e} \in I^{[p^e]} \text{ for some } c \in R^{\circ} \text{ and for all } e \gg 0\},$$
where $I^{[p^e]} := \{x^{p^e} \mid x \in I\}$ denotes the $e$-th Frobenius power of $I$.

An element $c\in R^\circ$ is called a {\it test element} if, for all ideals $I\subseteq R$, $x\in I^*$ if and only if $cx^{p^e}\in I^{[p^e]}$ for all $e\geq 0$. Test elements are known to exist when $R$ is a reduced excellent local ring \cite[Theorem 6.1]{HochsterHuneke2}. Fix a test element $c$. We say that $p^e$ is a {\it test exponent} of the pair $(c,I)$, if whenever $cx^{p^{e'}}\in I^{[p^{e'}]}$ for some $e'\geq e$, then $x\in I^*$. The existence of test exponent is closely related to the localization problem of tight closure, see \cite{HochsterHunekeTestExponents}. In this article, we need the following result of Sharp on the existence of test exponent that works simultaneously for all ideals generated by a system of parameters.

\begin{theorem}[{\cite[Corollary 2.4]{SharpTestExponents}}]
\label{thm: sharp text exponent}
Let $(R,\m)$ be a reduced and equidimensional excellent local ring of prime characteristic $p>0$ and let $c\in R^\circ$ be a test element. Then there exists $e>0$ such that $p^e$ is a test exponent of $(c, \fq)$ for every ideal $\fq\subseteq R$ that is generated by a system of parameters.
\end{theorem}

Let $(R,\m)$ be a Noetherian local ring of prime characteristic $p>0$. The {\it parameter test ideal} of $R$ can be defined by $$\tau_{\pa}(R) := \bigcap_{\fq} (\fq: \fq^*),$$ where $\fq$ runs over all ideals $\fq$ generated by a system of parameters. Clearly, any test element is contained in $\tau_{\pa}(R)$. We say $R$ is {\it F-rational} if $\tau_{\pa}(R) = R$, that is, $\fq^*=\fq$ for every ideal $\fq$ generated by a system of parameters. If $R$ is F-rational then $R$ is normal, and if additionally $R$ is excellent, then $R$ is Cohen-Macaulay, see \cite[Theorem 8.2]{HunekeSixLecutresCA}. In particular, F-rational rings of dimension at most one are regular. 

\subsection*{Frobenius action on local cohomology} Let $R$ be a Noetherian ring of prime characteristic $p>0$. For any ideal $I$ that is generated up to radical by $\underline{x} := x_1, \ldots, x_t$, the local cohomology module $H^i_I(R)$ can be defined as the cohomology of the \v{C}ech complex
$$
C^\bullet(\underline{x}; R):= 0 \to R \to \bigoplus_{i=1}^t R_{x_i} \to \cdots \to R_{x_1 \cdots x_t} \to 0.
$$
The Frobenius endomorphism $F:R \to R$ and its localizations thus induce a natural Frobenius action $F$ on $H^i_I(R)$ for each $i$. Of particular interest is the top local cohomology module, which can be described as
$$
H^t_I(R) \cong\varinjlim_n R/(x_1^n, \ldots, x_t^n),
$$
where the map in the system $\varphi_{n, m}: R/(x_1^n, \ldots, x_t^n) \to R/(x_1^m, \ldots, x_t^m)$ is the multiplication by $(x_1 \cdots x_t)^{m - n}$ for all $m \ge n$. Then for each $\overline{a} \in H^t_{I}(R)$, which is the canonical image of some $a+(x_1^n, \ldots, x_t^n)$, we find that $F(\overline{a})$ is the canonical image of $a^p +(x_1^{pn}, \ldots, x_t^{pn})$. We define the tight closure of the zero submodule of $H^t_I(R)$ as follows
$$0^*_{H^t_I(R)} = \{ \eta \mid cF^e(\eta)=0 \text{ for some } c \in R^{\circ} \text{ and for all } e \gg 0 \}.$$
Suppose $(R, \frak m)$ is an equidimensional excellent local ring of dimension $d$, and $x_1, \ldots, x_d$ is a system of parameters. It follows from the same argument as in \cite[Proposition 3.3]{SmithTightClosureParameterIdeal} that
\begin{equation}
\label{eqn: top lc}
0^*_{H^d_m(R)} \cong \varinjlim_n \frac{(x_1^n, \ldots, x_d^n)^*}{(x_1^n, \ldots, x_d^n)}.
\end{equation}
We see that $\tau_{\pa}(R)$ annihilates $0^*_{H^d_\m(R)}$. We also note that, when $R$ is Cohen-Macaulay, all the transition maps in the above direct limit are injective. It follows that an excellent local ring $(R,\m)$ is F-rational if and only if $R$ is Cohen-Macaulay and $0^*_{H_\m^d(R)}=0$. We refer the reader to \cite{SmithTightClosureParameterIdeal,SmithFrationalrings} for more details.

For a Noetherian domain $R$, the {\it absolute integral closure} of $R$, denoted by $R^+$, is the integral closure of $R$ in an algebraic closure of the fraction field of $R$. If $R$ is Noetherian and reduced, then we define $R^+:=\prod_{\fp \in \mathrm{Min}R} (R/\fp)^+$. We caution the reader that $R^+$ is rarely Noetherian as a ring. The {\it plus closure} of an ideal $I\subseteq R$ is defined to be $I^+:=IR^+\cap R$.  The following important results are proved by Hochster--Huneke and Smith respectively, for excellent local domains. But the statements can be immediately extended to the reduced and equidimensional setting under our definition of $R^+$.

\begin{theorem}[{\cite[Theorem 1.1]{HochsterHunekeBig}}]
\label{thm: HochsterHunekebigCM}
Let $(R,\m)$ be a reduced and equidimensional excellent local ring of prime characteristic $p>0$. Then $R^+$ is a big Cohen-Macaulay $R$-algebra, i.e., every system of parameters of $R$ is a regular sequence in $R^+$.
\end{theorem}

\begin{theorem}[{\cite[Theorem 5.1]{SmithTightClosureParameterIdeal}}]
\label{thm: SmithPlusClosure}
Let $(R,\m)$ be a reduced and equidimensional excellent local ring of prime characteristic $p>0$ and dimension $d$. Then for every ideal $\fq$ generated by a system of parameters, we have $\fq^*=\fq^+$ and that $0^*_{H_\m^d(R)}=\ker(H_\m^d(R)\to H_\m^d(R^+))$.
\end{theorem}

A Noetherian local ring $(R,\m)$ is called {\it F-injective} if the natural Frobenius action on $H_\m^i(R)$ is injective for every $i$. F-rational rings are always F-injective, and F-injective rings are reduced, see \cite{DattaMurayamaFinjective} or \cite[Chapter 4]{MaPolstraFsingularitiesBook}.

\subsection*{Buchsbaum and generalized Cohen-Macaulay modules} In this subsection we recall the basic definitions of Buchsbaum and generalized Cohen-Macaulay modules. We refer the reader to \cite{StuckradVogelBuchsbaumRingsBook} and \cite{TrungGeneralizedCM} for more details.

Let $(R,\m)$ be a Noetherian local ring and let $M$ be a finitely generated $R$-module of dimension $d$. Then $M$ is called {\it generalized Cohen-Macaulay} if $H_\m^i(M)$ has finite length for all $i<d$. A system of parameters $x_1,\dots,x_d$ of $M$ is called {\it standard} if
$$\length(M/\fq M)-e(\fq, M)=\sum_{i=0}^{d-1}\binom{d-1}{i}\length(H_\m^i(M))$$
where $\fq:=(x_1,\dots,x_d)$.
We say $M$ is a {\it Buchsbaum} $R$-module if every system of parameters of $M$ is standard. We say $R$ is a generalized Cohen-Macaulay ring (resp. Buchsbaum ring) if $R$ is a generalized Cohen-Macaulay module (resp. Buchsbaum module) over itself.

\begin{remark}
\label{rmk: generalized CM}
With notation as above, we have
\begin{enumerate}[(1)]
  \item Let $(R,\m)$ be homomorphic image of a Cohen-Macaulay ring (e.g., $R$ is excellent, see \cite[Corollary 1.2]{KawasakiArithmeticMacaulayfication}) and let $M$ be a finitely generated $R$-module of dimension $d$. Then $M$ is generalized Cohen-Macaulay if and only if $M_{\fp}$ is Cohen-Macaulay of dimension $d-\dim(R/\fp)$ for all $\fp\in\Spec(R)\backslash\{\m\}$. In particular, $R$ is generalized Cohen-Macaulay if and only if $R$ is equidimensional and $R_{\fp}$ is Cohen-Macaulay for all $\fp\in\Spec(R)\backslash\{\m\}$, see \cite[Satz 2.5 and Satz 3.8]{CuongSchenzelTrung} or \cite[Lemma 1.2 and Lemma 1.4]{TrungGeneralizedCM}.
  \item We have that $M$ is a generalized Cohen-Macaulay $R$-module if and only if $\sup \{\length(M/\fq M)-e(\fq, M)\}<\infty,$ where $\fq$ runs over all ideals generated by a system of parameters of $M$, see \cite[Satz 3.3]{CuongSchenzelTrung} or \cite[Lemma 1.1]{TrungGeneralizedCM}.
  \item Suppose $M$ is a generalized Cohen-Macaulay $R$-module of dimension $d$. Then for every ideal $\fq$ generated by a system of parameters of $M$, we have
  $$\length(M/\fq M)-e(\fq, M)\leq \sum_{i=0}^{d-1}\binom{d-1}{i}\length(H_\m^i(M))$$
  and equality holds when $\fq$ is contained in a sufficiently large power of $\m$ (so such $\fq$ is standard), see \cite[Satz 3.7]{CuongSchenzelTrung} or \cite[Lemma 1.5]{TrungGeneralizedCM}. 
\end{enumerate}
\end{remark}

We will need the following two criterions of Buchsbaum modules.

\begin{theorem}[{\cite[Proposition 3.2]{TrungGeneralizedCM}}]
\label{thm: finite criteion for Buchsbaum}
Let $(R,\m)$ be a Noetherian local ring and let $M$ be a finitely generated $R$-module of dimension $d$. Let $\{y_1,\dots,y_n\}$ be a fixed generating set of $\m$. Then $M$ is Buchsbaum if for any $d$ element subset $\{x_1,\dots,x_d\}$ of $\{y_1,\dots,y_n\}$, $x_1,\dots,x_d$ is a standard system of parameters of $M$.
\end{theorem}

\begin{theorem}[{\cite{StuckradMathNachr} or \cite[Theorem 3.4]{TrungGeneralizedCM}}]
\label{thm: surjectivity criterion for Buchsbaum}
Let $(R,\m)$ be a Noetherian local ring and let $M$ be a finitely generated $R$-module of dimension $d$. Then $M$ is Buchsbaum if and only if the natural homomorphism
$H^i(\m, M)\to H_\m^i(M)$ is surjective for all $i<d$, where $H^i(\m, M)$ denotes the $i$-th Koszul cohomology on any set of generators of $\m$.
\end{theorem}

\subsection*{Derived torsion functors and truncation}
Let $R$ be a Noetherian ring and let $I\subseteq R$ be an ideal. For an $R$-module $M$, we use $\mathbf{R}\Gamma_IM$ to denote the {\it derived $I$-power torsion} of $M$, often viewed as an object in $D(R)$, the derived category of $R$-modules. To obtain an explicit representative of $\mathbf{R}\Gamma_IM$, one can either apply the usual $I$-power torsion functor to an injective resolution of $M$, or use the \v{C}ech complex $C^\bullet(\underline{x}, R)\otimes M$, where $\underline{x}=x_1,\dots,x_n$ is any generating set of $I$ up to radical. Note that the $i$-th cohomology of $\mathbf{R}\Gamma_IM$ is the $i$-th local cohomology module of $M$ supported at $I$.

We assign to $D(R)$ the standard $t$-structure $(D(R)^{\leq 0}, D(R)^{\geq 0})$. We have the usual truncation functors $\tau^{\leq n}$ and $\tau^{\geq n}$ so that for any object $K\in D(R)$, there is a canonical triangle
$$\tau^{<n}K \to K \to \tau^{\geq n}K \xrightarrow{+1}.$$
With these notations, Schenzel proved the following criterion for Buchsbaum rings.

\begin{theorem}[{\cite[Theorem 2.3]{SchenzelApplicationsofDualizingComplex}}]
\label{thm: schenzel}
Let $(R,\m, k)$ be a Noetherian local ring of dimension $d$. Then $R$ is Buchsbaum if and only if $\tau^{<d}\mathbf{R}\Gamma_\m R$ is quasi-isomorphic to a complex of $k$-vector spaces (i.e., it comes from an object of $D(k)$).
\end{theorem}

In particular, Theorem \ref{thm: schenzel} implies immediately that if $(R,\m, k)$ is a Buchsbaum ring of dimension $d$, then $H_\m^i(R)$ is a $k$-vector space (i.e., annihilated by $\m$) for every $i<d$. We caution the reader that the converse is {\it not} true in general: there are many examples of Noetherian local rings $(R,\m, k)$ of dimension $d$ such that $H_\m^i(R)$ is a $k$-vector space for each $i<d$ but $R$ is not Buchsbaum, see \cite{GotoQuasiBuchsbaum}. Thus it is essential to work with a complex (or, in the derived category) in Schenzel's characterization of Buchsbaum rings.

Theorem \ref{thm: schenzel} has many applications. For example, using Theorem \ref{thm: schenzel}, it is proved in \cite{BhattMaSchwedeDualizingComplexFinjectiveDB} that F-injective generalized Cohen-Macualay rings are Buchsbaum (this was conjectured by Takagi and was first proved in \cite{MaFinjectiveBuchsbaum} using other methods). The full results in \cite{BhattMaSchwedeDualizingComplexFinjectiveDB} are stronger, and it yields a tight closure analog of the corresponding statement. To explain this, we recall that in \cite{BhattMaSchwedeDualizingComplexFinjectiveDB}, the {\it ${}^*$-truncation} (or {\it tight closure truncation}) of $\mathbf{R}\Gamma_\m R$ is defined as the object in $D(R)$ such that we have an exact triangle:
$$\tau^{<d,*}\mathbf{R}\Gamma_\m R \to \mathbf{R}\Gamma_\m R \to H_\m^d(R)/0^*_{H_\m^d(R)}[-d]\xrightarrow{+1}.$$
In particular, $h^i(\tau^{<d,*}\mathbf{R}\Gamma_\m R)=H_\m^i(R)$ for $i<d$ and $h^d(\tau^{<d,*}\mathbf{R}\Gamma_\m R)=0^*_{H_\m^d(R)}$.

\begin{theorem}[{\cite[Theorem 3.6]{BhattMaSchwedeDualizingComplexFinjectiveDB}}]
\label{thm: tight closure truncation}
Let $(R,\m, k)$ be an equidimensional excellent local ring of prime characteristic $p > 0$ and dimension $d$. Suppose $R$ is F-injective and $R_{\fp}$ is F-rational for all $\fp\in \Spec(R)\backslash\{\m\}$. Then $\tau^{<d, *}\mathbf{R}\Gamma_\m R$ is quasi-isomorphic to a complex of $k$-vector spaces (i.e., it comes from an object of $D(k)$).
\end{theorem}

We shall see in Remark \ref{rmk: tau radical F-inj} that the equivalence $(3)\Leftrightarrow(4)$ in our Main Theorem should be viewed as a generalization of Theorem \ref{thm: tight closure truncation}. For now, we record the following simple lemma which gives an alternative description of $\tau^{<d, *}\mathbf{R}\Gamma_\m R$.

\begin{lemma}
\label{lem: alternative description}
Let $(R,\m)$ be a reduced and equidmensional excellent local ring of prime characteristic $p > 0$ and dimension $d$. Then we have $\tau^{<d,*}\mathbf{R}\Gamma_\m R\cong \tau^{\leq d}(\mathbf{R}\Gamma_\m(R^+/R)[-1])$ in $D(R)$.
\begin{proof}
The short exact sequence $0\to R\to R^+\to R^+/R\to 0$ induces
$$\mathbf{R}\Gamma_\m(R^+/R)[-1]\to \mathbf{R}\Gamma_\m R\to \mathbf{R}\Gamma_\m R^+\xrightarrow{+1}.$$
Taking cohomology and noting that $H_\m^j(R^+)=0$ for all $j<d$ by Theorem \ref{thm: HochsterHunekebigCM}, we have
$$h^j(\mathbf{R}\Gamma_\m(R^+/R)[-1])=H_\m^j(R) \text{ for all $j<d$, and }$$
$$h^d(\mathbf{R}\Gamma_\m(R^+/R)[-1]) \cong \ker(H_\m^d(R)\to H_\m^d(R^+))\cong 0^*_{H_\m^d(R)}.$$
where the last isomorphism follows from Theorem \ref{thm: SmithPlusClosure}. In particular, the induced map
$$\tau^{\leq d}(\mathbf{R}\Gamma_\m(R^+/R)[-1]) \to \tau^{\leq d}\mathbf{R}\Gamma_\m R \cong \mathbf{R}\Gamma_\m R \to H_\m^d(R)/0^*_{H_\m^d(R)}[-d]$$
vanishes after taking the $d$-th cohomology, and hence is the zero map in $D(R)$ because $\tau^{\leq d}(\mathbf{R}\Gamma_\m(R^+/R)[-1])$ lives in cohomology degree $\leq d$ and $H_\m^d(R)/0^*_{H_\m^d(R)}[-d]$ lives in cohomology degree $d$. Thus we have an induced map
$$\tau^{\leq d}(\mathbf{R}\Gamma_\m(R^+/R)[-1]) \to \tau^{<d, *}\mathbf{R}\Gamma_\m R,$$
which is readily seen to be a quasi-isomorphism.
\end{proof}
\end{lemma}

\subsection*{F-finite rings and the $\Gamma$-construction}Let $R$ be a Noetherian ring of prime characteristic $p>0$ and let $F^e$: $R \to R, x\mapsto x^{p^e}$ denote the $e$-th iterated Frobenius endomorphism on $R$. To distinguish the source and target of the Frobenius, we adopt the commonly used notation $F^e_*R$ for the target of the Frobenius as a module over the source, that is, $F^e$: $R\to F^e_*R$. Under this notation, elements in $F^e_*R$ are denoted by $F^e_*r$ where $r\in R$, and the $R$-module structure on $F^e_*R$ is defined via $r_1\cdot F^e_*r_2=F^e_*(r_1^{p^e}r_2)$. $R$ is called {\it F-finite} if $F^e_*R$ is a finitely generated $R$-module for some (or equivalently, all) $e>0$. It is well-known that a Noetherian F-finite ring is excellent and is a homomorphic image of an F-finite regular ring, see \cite[Theorem 2.5]{Kunz2}, \cite[Remark 13.6]{Gabbertstructure}, and \cite[Chapter 10]{MaPolstraFsingularitiesBook}.

We will need Hochster--Huneke's $\Gamma$-construction \cite[Section 6]{HochsterHuneke2} to pass from a complete Noetherian local ring to an F-finite local ring. We briefly recall the construction here. Let $(R,\m, k)$ be a complete Noetherian local ring with coefficient field $k$ of characteristic $p>0$. We fix a $p$-base $\Lambda$ of $k$ and let $\Gamma\subseteq \Lambda$ be a cofinite subset. We denote by $k^{\Gamma, e}$ the purely inseparable extension of $k$ by adjoining all $p^e$-th roots of elements in $\Gamma$. Set $$R^{\Gamma, e}:= R\widehat{\otimes}_kk^{\Gamma, e}, \text{ and } R^\Gamma:=\bigcup_e R^{\Gamma, e}.$$ It is readily seen that $R\to R^\Gamma$ is faithfully flat, purely inseparable, and that $\m R^\Gamma$ is the unique maximal ideal of $R^\Gamma$. We will use the following results.

\begin{lemma}[{\cite[Lemma 6.6 and Lemma 6.13]{HochsterHuneke2}}]
\label{lem: gamma construction reduced}
Let $(R,\m)$ be a complete Noetherian local ring of prime characteristic $p > 0$. Then $R^\Gamma$ is F-finite for all cofinite $\Gamma\subseteq \Lambda$. If, in addition, $R$ is reduced, then for all sufficiently small choices of $\Gamma$ cofinal in $\Lambda$, $R^\Gamma$ is reduced.
\end{lemma}

\begin{lemma}[{\cite[Lemma 4.3 and Lemma 4.4]{PhamHungQuyTightClosureParameterIdeal}}]
\label{lem: gamma construction F-rat}
Let $(R,\m)$ be an equidimensional complete Noetherian local ring of prime characteristic $p > 0$ and dimension $d$ such that $R_{\fp}$ is F-rational for all $\fp\in\Spec(R)\backslash\{\m\}$. Then for all sufficiently small choices of $\Gamma$ cofinal in $\Lambda$, $(\fq R^\Gamma)^*=\fq^*R^\Gamma$ for all ideals $\fq$ that are generated by a system of parameters, and that $0^*_{H_\m^d(R^\Gamma)}=0^*_{H_\m^d(R)}\otimes_RR^\Gamma$.
\end{lemma}

\section{Tight closure and limit closure}

The goal of this section is to explain $(1)\Leftrightarrow(2)$ in the Main Theorem. We deduce it using a series of results on limit closure and the following result of Goto--Nakamura \cite{GotoNakamuraBoundofDifference}, which can be viewed as a tight closure analog of the generalized Cohen-Macaulay property (this theorem also partially motivates our work in this article, as mentioned earlier).

\begin{theorem}[{\cite[Theorem 1.1]{GotoNakamuraBoundofDifference}}]
\label{thm: GotoNakamura}
Let $(R,\m)$ be an equidimensional excellent local ring of prime characteristic $p>0$ and dimension $d$. Then the following conditions are equivalent.
\begin{enumerate}[(1)]
  \item $\sup\{\length(\fq^*/\fq)\}<\infty$, where $\fq$ runs over all ideals generated by a system of parameters.
  \item $R_{\fp}$ is F-rational for every $\fp\in\Spec(R)\backslash\{\m\}$.
  \item $R$ is generalized Cohen-Macaulay and $0^*_{H_\m^d(R)}$ has finite length.
\end{enumerate}
Moreover, when this is the case, we have that $\sup\{e(\fq, R)-\length(R/\fq^*)\}<\infty$,\footnote{In fact, when $R$ is unmixed, this is equivalent to $(1)-(3)$ of the theorem, see Remark \ref{rmk: pseudo generalized CM} (3).} where $\fq$ runs over all ideals generated by a system of parameters.
\end{theorem}

\subsection*{Limit closure}Let $(R,\m)$ be a Noetherian local ring of dimension $d$ and let $x_1, \ldots, x_d$ be a system of parameters. The {\it limit closure} of $\fq:=(x_1, \ldots, x_d)$ in $R$ is defined as
\[
\fq^{\lim} := \bigcup_{n\geq 0} \left((x_1^{n+1}, \ldots, x_d^{n + 1}):_R (x_1\cdots x_d)^n\right).
\]
We note that $\fq^{\lim}/\fq$ is the kernel of the natural map $R/\fq \to H_\m^d(R)$. In particular, $\fq^{\lim}$ is well-defined: it does not depend on the choice of its generators $x_1,\dots,x_d$. Moreover, by (\ref{eqn: top lc}), we know that
\begin{equation}\label{eqn: top lc lim closure}
0^*_{H^d_m(R)} \cong \varinjlim_n \frac{(x_1^n, \ldots, x_d^n)^*}{(x_1^n, \ldots, x_d^n)^{\lim}}\cong \varinjlim_e \frac{(\fq^{[p^e]})^*}{(\fq^{[p^e]})^{\lim}}.
\end{equation}
where the transition maps in each direct limit system are all injective.

\begin{remark}
\label{rmk: lim closure tight closure}
Let $(R,\m)$ be an equidimensional excellent local ring of prime characteristic $p>0$ and let $\fq\subseteq R$ be an ideal that is generated by a system of parameters. Then we have $\fq^{\lim}\subseteq \fq^*$, see \cite[Theorem 2.3]{HunekeSixLecutresCA}. It follows that we have inequalities:
$$\length(R/\fq)\geq e(\fq, R)\geq \length(R/\fq^{\lim})\geq \length(R/\fq^*),$$
where the second inequality follows from \cite[Lemma 2.3]{CuongHoaLoan}, see also \cite[Theorem 9]{MaPhamSmirnovColengthMultiplicity} (note that the first two inequalities above hold for any Noetherian local ring). Moreover, under mild assumptions on $R$, certain equalities in the above chain of inequalities characterize $R$ being Cohen-Macaulay or F-rational. We refer the reader to \cite{MaPhamSmirnovColengthMultiplicity} (and the reference therein) for details towards this direction.
\end{remark}

The following lemma is well-known to experts (it is implicit in \cite{GotoOntheAssociatedGradedRingsofBuchsbaumRings} and \cite{CuongHoaLoan}). Since we could not find a good reference, we include a short argument.
\begin{lemma}
\label{lem: length of q^lim/q}
Let $(R,\m)$ be a generalized Cohen-Macaulay ring of dimension $d$. Then for every ideal $\fq$ generated by a system of parameters, we have
\begin{equation}\label{eqn: q^lim/q}
\length(\fq^{\lim}/\fq)\leq \sum_{i=0}^{d-1}\binom{d}{i}\length(H_\m^i(R)),
\end{equation}
with equality if and only if $\fq$ is standard.
\end{lemma}
\begin{proof}
We have $\length(\fq^{\lim}/\fq)=\length(R/\fq)-e(\fq, R)+ e(\fq, R)-\length(R/\fq^{\lim})$. By Remark \ref{rmk: generalized CM} (3),
\begin{equation}\label{eqn: l(q)-e(q)}
\length(R/\fq)-e(\fq, R)\leq \sum_{i=0}^{d-1}\binom{d-1}{i}\length(H_\m^i(R)).
\end{equation}
By \cite[Corollary 4.3]{CuongHoaLoan} and \cite[Theorem 3.7]{SharpHamiehLengthsGeneralizedFractions}, we have
\begin{equation}
\label{eqn: e(q)-q^lim}
e(\fq, R)-\length(R/\fq^{\lim})\leq \sum_{i=0}^{d-1}\binom{d-1}{i-1}\length(H_\m^i(R))
\end{equation}
and equality holds when $\fq$ is standard by \cite[Theorem 5.1]{CuongHoaLoan}. Putting (\ref{eqn: l(q)-e(q)}) and (\ref{eqn: e(q)-q^lim}) together we obtain (\ref{eqn: q^lim/q}). If equality occurs in (\ref{eqn: q^lim/q}), then equality also occurs in (\ref{eqn: l(q)-e(q)}) and thus $\fq$ is standard by definition. Finally, if $\fq$ is standard then we have equality in both (\ref{eqn: l(q)-e(q)}) and (\ref{eqn: e(q)-q^lim}) and thus we have equality in (\ref{eqn: q^lim/q}) as well.
\end{proof}

\begin{remark}
\label{rmk: limit closure Buchsbaum}
Let $(R,\m)$ be a Noetherian local ring and let $\fq$ denote an ideal generated by a system of parameters. If $\length(\fq^{\lim}/\fq)$ is independent of $\fq$, then as $$\length(\fq^{\lim}/\fq) = (\length(R/\fq)-e(\fq, R))+(e(\fq, R)-\length(R/\fq^{\lim}))\geq \length(R/\fq)-e(\fq, R),$$ we know that $\{\length(R/\fq)-e(\fq, R)\}$ is bounded (where $\fq$ runs over all ideals generated by a system of parameters). Thus $R$ is generalized Cohen-Macaulay by Remark \ref{rmk: generalized CM} (2). It follows that every $\fq$ that is contained in a sufficiently large power of $\m$ is standard (see Remark \ref{rmk: generalized CM} (3)), and thus we have equality in Lemma \ref{lem: length of q^lim/q} for all such $\fq$, but then we have equality in Lemma \ref{lem: length of q^lim/q} for all $\fq$ since $\length(\fq^{\lim}/\fq)$ is independent of $\fq$. It follows that every $\fq$ generated by a system of parameters is standard by Lemma \ref{lem: length of q^lim/q} again, and thus $R$ is Buchsbaum.
\end{remark}

\begin{remark}
\label{rmk: pseudo generalized CM}
The difference $e(\fq,R)-\length(R/\fq^{\lim})$ has been studied extensively in \cite{CuongHoaLoan,CuongNhanPseudoCM,CuongLoanPseudoBuchsbaumJapan}.
\begin{enumerate}[(1)]
  \item A Noetherian local ring $(R,\m)$ such that $\sup\{e(\fq,R)-\length(R/\fq^{\lim})\}<\infty$ (where $\fq$ runs over all ideals generated by a system of parameters) is called {\it pseudo generalized Cohen-Macaulay} in \cite{CuongNhanPseudoCM}. It follows from \cite[Corollary 3.3]{CuongNhanPseudoCM} that if $(R,\m)$ is pseudo generalized Cohen-Macaulay and $\widehat{R}$ is unmixed, then $R$ is generalized Cohen-Macaulay.
  \item Similarly, a Noetherian local ring $(R,\m)$ such that $e(\fq,R)-\length(R/\fq^{\lim})$ is a constant independent of $\fq$ is called {\it pseudo Buchsbaum} in \cite{CuongLoanPseudoBuchsbaumJapan}. It is proved in \cite[Theorem 1.1]{CuongLoanPseudoBuchsbaumJapan} that if $(R,\m)$ is pseudo Buchsbaum and $\widehat{R}$ is unmixed, then $R$ is Buchsbaum.
  \item Now suppose $(R,\m)$ an unmixed excellent local ring of prime characteristic $p>0$ (thus $\widehat{R}$ is also unmixed). If $\sup\{e(\fq, R)-\length(R/\fq^*)\}<\infty$, then as $\fq^{\lim}\subseteq \fq^*$ by Remark \ref{rmk: lim closure tight closure}, we have $\sup\{e(\fq, R)-\length(R/\fq^{\lim})\}<\infty$ and thus $R$ is generalized Cohen-Macaulay. But then $R_{\fp}$ is F-rational for all $\fp\in \Spec(R)\backslash\{\m\}$ by \cite[Theorem 1.2]{GotoNakamuraBoundofDifference}.
\end{enumerate}
\end{remark}

Now we prove the main result of this section.

\begin{proposition}
\label{prop: (1) equiv (2) in main}
Let $(R,\m)$ be an equidimensional excellent local ring of prime characteristic $p>0$ and dimension $d$. Let $\fq$ denote an ideal generated by a system of parameters. Consider the following conditions
\begin{enumerate}[(1)]
  \item The difference $e(\fq, R)-\length(R/\fq^*)$ is independent of $\fq$.
  \item $\length(\fq^*/\fq)$ is independent of $\fq$.
\end{enumerate}
Then we have $(2)\Rightarrow(1)$, if additionally $R$ is unmixed, then we also have $(1)\Rightarrow(2)$.
\end{proposition}
\begin{proof}
Suppose $(2)$ holds. By Theorem \ref{thm: GotoNakamura}, we know that $R$ is generalized Cohen-Macaulay and $0^*_{H_\m^d(R)}$ has finite length. By Remark \ref{rmk: lim closure tight closure}, we have $\fq^{\lim}\subseteq \fq^*$. It follows that
$$\ell(\fq^*/\fq)=\ell(\fq^{\lim}/\fq)+\ell(\fq^*/\fq^{\lim}).$$
Since $R$ is generalized Cohen-Macaulay, by Lemma \ref{lem: length of q^lim/q}, for all $e\gg0$ we have
\begin{equation}
\label{eqn: first inequality}
\length(\fq^{\lim}/\fq)\leq \sum_{i=0}^{d-1}\binom{d}{i}\length(H_\m^i(R)) = \ell((\fq^{[p^e]})^{\lim}/\fq^{[p^e]})
\end{equation}
where the equality follows from the fact that $\fq^{[p^e]}$ is standard, see Remark \ref{rmk: generalized CM} (3) and Lemma \ref{lem: length of q^lim/q}. Moreover, we have
$\fq^*/\fq^{\lim}\hookrightarrow 0^*_{H_\m^d(R)},$
and since $0^*_{H_\m^d(R)}$ has finite length, it follows from (\ref{eqn: top lc lim closure}) that
$(\fq^{[p^e]})^*/(\fq^{[p^e]})^{\lim} \cong 0^*_{H_\m^d(R)}$ for all $e\gg0$. In particular, for all $e\gg0$ we have
\begin{equation}
\label{eqn: second inequality}
\ell(\fq^*/\fq^{\lim}) \leq \ell(0^*_{H_\m^d(R)})=\ell((\fq^{[p^e]})^*/(\fq^{[p^e]})^{\lim})
\end{equation}
Now we note that for all $e\gg0$, we have
$$\ell(\fq^{\lim}/\fq)+\ell(\fq^*/\fq^{\lim})= \ell(\fq^*/\fq) = \ell((\fq^{[p^e]})^{*}/\fq^{[p^e]}) =\ell((\fq^{[p^e]})^{\lim}/\fq^{[p^e]})+\ell((\fq^{[p^e]})^*/(\fq^{[p^e]})^{\lim})$$
where the equality in the middle follows from our assumption that $\ell(\fq^*/\fq)$ is independent of $\fq$. This combined with (\ref{eqn: first inequality}) and (\ref{eqn: second inequality}) shows that both $\ell(\fq^{\lim}/\fq)$ and $\ell(\fq^*/\fq^{\lim})$ are independent of $\fq$. In particular, $R$ is Buchsbaum by Remark \ref{rmk: limit closure Buchsbaum} and thus every $\fq$ is standard. But then we have
$$e(\fq, R)-\length(R/\fq^*)= (e(\fq, R)-\length(R/\fq)) + \length(\fq^*/\fq)$$
is independent of $\fq$, so $(1)$ holds.

Next we suppose $(1)$ holds and $R$ is unmixed (note that this implies $\widehat{R}$ is unmixed since $R$ is excellent). By Remark \ref{rmk: lim closure tight closure}, we have $\fq^{\lim}\subseteq \fq^*$. Thus we have
$$e(\fq,R)-\length(R/\fq^{*})= (e(\fq,R)-\length(R/\fq^{\lim})) + \length(\fq^*/\fq^{\lim}).$$
Since $e(\fq,R)-\length(R/\fq^{*})$ is independent of $\fq$, we know that $\sup\{e(\fq,R)-\length(R/\fq^{\lim})\}<\infty$ and $\sup\{\length(\fq^*/\fq^{\lim})\}<\infty$. Thus by Remark \ref{rmk: pseudo generalized CM} (1), $R$ is generalized Cohen-Macaulay. Now by \cite[Corollary 4.3]{CuongHoaLoan} and \cite[Theorem 3.7]{SharpHamiehLengthsGeneralizedFractions}, we have
$$
e(\fq, R)-\length(R/\fq^{\lim})\leq \sum_{i=0}^{d-1}\binom{d-1}{i-1}\length(H_\m^i(R))
$$
with equality holds when $\fq$ is standard by \cite[Theorem 5.1]{CuongHoaLoan} (in particular, when $\fq$ is contained in a sufficiently high power of $\m$ by Remark \ref{rmk: generalized CM} (3)). Moreover, as $\sup\{\length(\fq^*/\fq^{\lim})\}<\infty$, by examining (\ref{eqn: top lc lim closure}), we know that
$(\fq^{[p^e]})^*/(\fq^{[p^e]})^{\lim} \cong 0^*_{H_\m^d(R)}$ for all $e\gg0$.

Therefore, if $e(\fq,R)-\length(R/\fq^{*})$ is independent of $\fq$, then both $e(\fq, R)-\length(R/\fq^{\lim})$ and $\length(\fq^*/\fq^{\lim})$ are independent of $\fq$. Now by Remark \ref{rmk: pseudo generalized CM} (2), $R$ is Buchsbaum. But then
$$\length(\fq^*/\fq)= (\length(R/\fq)-e(\fq, R)) + (e(\fq, R)-\length(R/\fq^*))$$
is independent of $\fq$, so $(2)$ holds.
\end{proof}

\begin{remark}
\label{rmk: length indep implies Buchsbaum}
It is evident from the proof of Proposition \ref{prop: (1) equiv (2) in main} that, if $(R,\m)$ is an unmixed excellent local ring of prime characteristic $p>0$ and if either $(1)$ or $(2)$ in Proposition \ref{prop: (1) equiv (2) in main} holds, then $R$ is Buchsbaum and that $\fq^*/\fq^{\lim}\cong 0^*_{H_\m^d(R)}$ for every ideal $\fq$ that is generated by a system of parameters.
\end{remark}

\section{Relative Frobenius action on local cohomology}

In this section, we study $\length(\fq^*/\fq)$ when $\fq$ is a parameter ideal contained in $\tau_{\pa}(R)$. Our full result, Theorem \ref{thm: length is independent of sop}, answers positively a question of the second author \cite[Question 1]{PhamHungQuyTightClosureParameterIdeal} and it immediately implies one implication ($(3)\Rightarrow(2)$) in the Main Theorem. Our main technique is the {\it relative Frobenius action} on local cohomology introduced in \cite{PolstraPhamHungQuyNilpotence}, which turns out to be very useful in the study of tight closure of parameter ideals. 

Let $R$ be a Noetherian ring of prime characteristic $p>0$ and let $J \subseteq I$ be ideals of $R$. The Frobenius endomorphism $F:R/J\to R/J$ can be factored as a composition of two natural maps:
$$R/ J  \to R/J^{[p]} \twoheadrightarrow R/J,$$
where the second map is the natural projection and we denote the first map by $F_R$, which sends $r + J$ to $r^p + J^{[p]}$ for all $r \in R$. The homomorphism $F_R$ induces the relative Frobenius actions on local cohomology modules $F_R: H^i_I(R/J) \to H^i_I(R/J^{[p]})$. We define the {\it relative tight closure} of the zero submodule of $H^i_I(R/J)$ as
$$0^{*_R}_{H^i_I(R/J)} := \{ \eta \mid c F^e_R(\eta)=0 \in H^i_I(R/J^{[p^e]}) \text{ for some } c \in R^{\circ} \text{ and for all } e \gg 0 \}.$$
It is easy to verify that, if $J$ is an $\m$-primary ideal, then $0^{*_R}_{H^0_\m(R/J)}=J^*/J$, while if $J=0$ then $0^{*_R}_{H^d_\m(R/J)}=0^{*}_{H^d_\m(R)}$. Roughly speaking, we will use relative tight closure as a bridge to connect tight closure of parameter ideals and tight closure of the zero submodule in the top local cohomology module.
\begin{remark}
\label{rmk: relative tight closure in top LC}
If $(R, \frak m)$ is an equidimensional excellent local ring and $J = (y_1, \ldots, y_s)$ such that $y_1, \ldots, y_s, x_1, \ldots, x_t$ is part of a system of parameters, then with $I=(y_1,\dots, y_s, x_1,\dots,x_t)$, we have by \cite[Lemma 5.7]{PolstraPhamHungQuyNilpotence} that
\[
0^{*_R}_{H^t_I(R/J)} \cong \varinjlim_n \frac{(J, x_{1}^n, \ldots, x_{t}^n)^*}{(J, x_{1}^n, \ldots, x_{t}^n)}
\]
where the transition maps are multiplication by $x_1\cdots x_t$.
\end{remark}

The following is our key technical lemma.

\begin{lemma}
\label{lem: key lemma}
Let $(R, \frak m)$ be an equidimensional excellent local ring of dimension $d$, and $x_1, \ldots, x_d$ a system of parameters. Set $\fq_i = (x_1, \ldots, x_i)$ for $i \le d$. Then for each $i<d$ the short exact sequence
$$0 \to R/(\fq_i : x_{i+1}) \overset{\cdot x_{i+1}}{\longrightarrow} R/\fq_i \longrightarrow R/\fq_{i+1} \to 0$$
induces an exact sequence
\begin{equation}
\label{eqn: LES relative tight closure}
\cdots \to  H^{d-i-1}_{\frak m}(R/\fq_i) \to 0^{*_R}_{H^{d-i-1}_{\frak m}(R/\fq_{i+1})} \to 0^{*_R}_{H^{d-i}_{\frak m}(R/\fq_i)} \to x_{i+1}0^{*_R}_{H^{d-i}_{\frak m}(R/\fq_i)} \to 0,
\end{equation}
and an injective map
$$\frac{H^{d-i-1}_{\frak m}(R/\fq_{i+1})}{0^{*_R}_{H^{d-i-1}_{\frak m}(R/\fq_{i+1})}} \hookrightarrow \frac{H^{d-i}_{\frak m}(R/\fq_i)}{0^{*_R}_{H^{d-i}_{\frak m}(R/\fq_i)}}.$$
\end{lemma}
\begin{proof}
We first note that since $\dim(R/\fq_i) = d-i$ and $\fq_{i+1} \subseteq \Ann((\fq_i:x_{i+1})/\fq_i)$, we have that $\dim ((\fq_i:x_{i+1})/\fq_i) \le d-i-1$ and thus $H^{d-i}_{\frak m}(R/(\fq_i:x_{i+1})) \cong H^{d-i}_{\frak m}(R/\fq_i)$. The short exact sequence
$$0 \to R/(\fq_i : x_{i+1}) \overset{\cdot x_{i+1}}{\longrightarrow} R/\fq_i \longrightarrow R/\fq_{i+1} \to 0$$
induces an exact sequence of local cohomology
$$\cdots \to  H^{d-i-1}_{\frak m}(R/\fq_i) \overset{\alpha}{\longrightarrow} H^{d-i-1}_{\frak m}(R/\fq_{i+1}) \overset{\beta}{\longrightarrow} H^{d-i}_{\frak m}(R/\fq_i) \overset{\cdot x_{i+1}}{\longrightarrow} H^{d-i}_{\frak m}(R/\fq_i) \to 0.$$
For each $e \ge 1$ we have the following commutative diagram
\[
\begin{CD}
H^{d-i-1}_{\frak m}(R/\fq_i) @>\alpha>> H^{d-i-1}_{\frak m}(R/\fq_{i+1}) @>\beta>> H^{d-i}_{\frak m}(R/\fq_i) @>\cdot x_{i+1}>> H^{d-i}_{\frak m}(R/\fq_i) \\
@VVF^{e}_RV @VVF^{e}_RV @VVF^{e}_RV @VVF^{e}_RV\\
H^{d-i-1}_{\frak m}(R/\fq_i^{[p^e]}) @>\alpha_e>> H^{d-i-1}_{\frak m}(R/\fq_{i+1}^{[p^e]}) @>\beta_e>> H^{d-i}_{\frak m}(R/\fq_i^{[p^e]}) @>\cdot x_{i+1}^{p^e}>> H^{d-i}_{\frak m}(R/\fq_i^{[p^e]}). \\
\end{CD}
\]

Set $\frak a(R) = \prod_{i=0}^{d-1} \Ann(H^i_{\frak m}(R))$. We have $\dim (R/\frak a(R))<d$ and by \cite[Corollary 4.4]{PhamHungQuyUniformAnnihilators}, there exists $N$ such that $$\frak a(R)^N H^{d-i-1}_{\frak m}(R/\fq_i^{[p^e]}) =0$$ for all $e\geq 0$. 
Thus, as $R$ is equidimensional, we can choose $c \in \frak a(R)^N \cap R^{\circ}$ and we have $c F^e_R \circ \alpha = \alpha_e (cF^e_R) = 0$ for all $e \ge 0$. It follows that $\mathrm{Im}(\alpha) \subseteq 0^{*_R}_{H^{d-i-1}_{\frak m}(R/\fq_{i+1})}$. Furthermore, we have $\beta (0^{*_R}_{H^{d-i-1}_{\frak m}(R/\fq_{i+1})}) \subseteq 0^{*_R}_{H^{d-i}_{\frak m}(R/\fq_i)}$ which follows from the middle square of the above commutative diagram. Therefore we have an exact sequence
$$H^{d-i-1}_{\frak m}(R/\fq_i) \to 0^{*_R}_{H^{d-i-1}_{\frak m}(R/\fq_{i+1})} \to 0^{*_R}_{H^{d-i}_{\frak m}(R/\fq_i)}.$$
To obtain the exact sequence (\ref{eqn: LES relative tight closure}) it suffices to show
$$\beta(0^{*_R}_{H^{d-i-1}_{\frak m}(R/\fq_{i+1})}) = 0:_{0^{*_R}_{H^{d-i}_{\frak m}(R/\fq_i)}}x_{i+1} = (0:_{H^{d-i}_{\frak m}(R/\fq_i)}x_{i+1}) \cap 0^{*_R}_{H^{d-i}_{\frak m}(R/\fq_i)}.$$
Hence it is enough to show
$$\beta(0^{*_R}_{H^{d-i-1}_{\frak m}(R/\fq_{i+1})}) = \beta(H^{d-i-1}_{\frak m}(R/\fq_{i+1})) \cap 0^{*_R}_{H^{d-i}_{\frak m}(R/\fq_i)}.$$
This will follow from the last assertion about the injectivity of the induced homomorphism
\begin{equation}
\label{eqn: injectivity of induced map}
\frac{H^{d-i-1}_{\frak m}(R/\fq_{i+1})}{0^{*_R}_{H^{d-i-1}_{\frak m}(R/\fq_{i+1})}} \overset{\beta}{\rightarrow} \frac{H^{d-i}_{\frak m}(R/\fq_i)}{0^{*_R}_{H^{d-i}_{\frak m}(R/\fq_i)}}.
\end{equation}
Thus it remains to establish that (\ref{eqn: injectivity of induced map}) is an injection. By Remark \ref{rmk: relative tight closure in top LC} we have
$$\frac{H^{d-i-1}_{\frak m}(R/\fq_{i+1})}{0^{*_R}_{H^{d-i-1}_{\frak m}(R/\fq_{i+1})}} \cong \varinjlim_n \frac{R}{(\fq_{i+1}, x_{i+2}^n, \ldots, x_{d}^n)^*},$$
and that
$$\frac{H^{d-i}_{\frak m}(R/\fq_i)}{0^{*_R}_{H^{d-i}_{\frak m}(R/\fq_i)}} \cong \varinjlim_n \frac{R}{(\fq_i, x_{i+1}^n, \ldots, x_{d}^n)^*}.
$$
Moreover, one checks that the map $\beta$ is induced by multiplication by $x_{i+1}^{n-1}$ map
$$\frac{R}{(\fq_{i+1}, x_{i+2}^n, \ldots, x_{d}^n)^*}\xrightarrow{\cdot x_{i+1}^{n-1}} \frac{R}{(\fq_i, x_{i+1}^n, \ldots, x_{d}^n)^*}  $$
and taking a direct limit for all $n$. But for each $n$ we have
$$(\fq_i, x_{i+1}^n, \ldots, x_{d}^n)^* : x_{i+1}^{n-1} = (\fq_{i+1}, x_{i+2}^n, \ldots, x_{d}^n)^*$$
by \cite[Theorem 2.3]{AberbachHunekeSmith}.  Thus the above multiplication by $x_{i+1}^{n-1}$ map is injective for each $n$. Therefore the direct limit map (\ref{eqn: injectivity of induced map}) is also injective. This finishes the proof.
\end{proof}

Now we can prove the main result of this section.

\begin{theorem}
\label{thm: length is independent of sop}
Let $(R, \frak m)$ be an equidimensional excellent local ring of dimension $d$ such that $\tau_{\pa}(R)$ is $\m$-primary. Let $\fq$ be an ideal generated by a system of parameters that is contained in  $\tau_{\pa}(R)$. Then we have
$$\ell(\fq^*/\fq) = \sum_{i=0}^{d-1}\binom{d}{i} \ell(H^i_{\frak m}(R)) + \ell(0^*_{H^d_{\frak m}(R)}).$$
\end{theorem}
\begin{proof}
First note that, since $\tau_{\pa}(R)$ is $\m$-primary, $R_{\fp}$ is F-rational for every $\fp\in\Spec(R)\backslash\{\m\}$ and thus $R$ is generalized Cohen-Macaulay and $\ell(0^*_{H^d_{\frak m}(R)})<\infty$ by Theorem \ref{thm: GotoNakamura}.

By Remark \ref{rmk: lim closure tight closure}, we have $\fq^{\lim}\subseteq \fq^*$. It follows that
$$\ell(\fq^*/\fq)=\ell(\fq^{\lim}/\fq)+\ell(\fq^*/\fq^{\lim}).$$
Since $\fq\in \tau_{\pa}(R)$, by \cite[Remark 5.11]{HunekeSixLecutresCA} and \cite[Corollary 6.18]{GotoYamagishi}, we know that $\fq$ is a standard system of parameters. Thus by Lemma \ref{lem: length of q^lim/q} we have
$$\ell(\fq^{\lim}/\fq)=  \sum_{i=0}^{d-1}\binom{d}{i} \ell(H^i_{\frak m}(R)).$$
Therefore it is enough to show that $\ell(\fq^*/\fq^{\lim})=\ell(0^*_{H^d_{\frak m}(R)})$. We have natural maps
$$\fq^*/\fq \twoheadrightarrow \fq^*/\fq^{\lim} \hookrightarrow 0^*_{H^d_{\frak m}(R)}.$$
It suffices to show that $\fq^*/\fq\to 0^*_{H^d_{\frak m}(R)}$ is surjective (for then it implies $\fq^*/\fq^{\lim}\cong 0^*_{H^d_{\frak m}(R)}$).
Let $\fq=(x_1,\dots,x_d)$ and let $\fq_i=(x_1,\dots,x_i)$. By Remark \ref{rmk: relative tight closure in top LC}, we know that
$$0^{*_R}_{H^{d-i}_{\frak m}(R/\fq_i)}\cong \varinjlim_n \frac{(\fq_i, x_{i+1}^n, \ldots, x_{d}^n)^*}{(\fq_i, x_{i+1}^n, \ldots, x_{d}^n)}.$$
Since $x_{i+1}\in \fq\subseteq \tau_{\pa}(R)$, we have $x_{i+1}0^{*_R}_{H^{d-i}_{\frak m}(R/\fq_i)}=0$ for each $i$. Now we apply Lemma \ref{lem: key lemma} to obtain that
$$0^{*_R}_{H^{d-i-1}_{\frak m}(R/\fq_{i+1})} \twoheadrightarrow 0^{*_R}_{H^{d-i}_{\frak m}(R/\fq_i)}$$
is surjective for each $i$. It follows that the following composition
$$\fq^*/\fq\cong 0^{*_R}_{H^{0}_{\frak m}(R/\fq_{d})} \twoheadrightarrow 0^{*_R}_{H^{1}_{\frak m}(R/\fq_{d-1})} \twoheadrightarrow \cdots  \twoheadrightarrow 0^{*_R}_{H^{d-i}_{\frak m}(R/\fq_i)}\twoheadrightarrow  \cdots \twoheadrightarrow  0^{*_R}_{H^{d}_{\frak m}(R/\fq_0)} \cong 0^{*}_{H^{d}_{\frak m}(R)} $$
is surjective. This finishes the proof.
\end{proof}

\section{The main result}

In this section, we prove our Main Theorem.

\begin{theorem}
\label{thm: main theorem}
Let $(R,\m, k)$ be an unmixed excellent local ring of prime characteristic $p>0$ and dimension $d\geq 1$. Let $\fq$ denote an ideal generated by a system of parameters. Then the following conditions are equivalent:
\begin{enumerate}[(1)]
  \item The difference $e(\fq, R)-\length(R/\fq^*)$ is independent of $\fq$.
  \item $\length(\fq^*/\fq)$ is independent of $\fq$.
  \item $\m\fq^*\subseteq \fq$ for every $\fq$, that is, $\tau_{\pa}(R)$ contains $\m$.
  \item The ${}^*$-truncation $\tau^{<d,*}\mathbf{R}\Gamma_\m R$ is quasi-isomorphic to a complex of $k$-vector spaces.
\end{enumerate}
\end{theorem}

\begin{remark}
If $\dim(R)=0$, then conditions $(1)$ and $(2)$ above are empty conditions while conditions $(3)$ and $(4)$ are simply saying that the nilradical of $R$ is annihilated by $\m$.
\end{remark}

\begin{proof}[Proof of Theorem \ref{thm: main theorem}]
First of all, $(1)\Leftrightarrow(2)$ follows from Proposition \ref{prop: (1) equiv (2) in main} and $(3)\Rightarrow(2)$ follows from Theorem \ref{thm: length is independent of sop}. We will show that $(2)\Rightarrow (4) \Rightarrow (3)$ below.

\begin{proof}[Proof of $(2)\Rightarrow (4)$] Suppose $\length(\fq^*/\fq)$ is independent of $\fq$, by Theorem \ref{thm: GotoNakamura}, we know that $R_{\fp}$ is F-rational for all primes $\fp\in\Spec(R)\backslash\{\m\}$. In particular, $R_{\fp}$ is regular for all minimal primes $\fp$ of $R$. Since $R$ is unmixed, it follows that $R$ is reduced and equidimensional.

Next we note that by Remark \ref{rmk: length indep implies Buchsbaum}, we know that $R$ is Buchsbaum. Since $R$ is excellent, we know that $\fq^*\widehat{R}=(\fq\widehat{R})^*$ (see \cite[Proposition 1.5]{HunekeSixLecutresCA}) and that $\widehat{R}_{\fp}$ is F-rational for all $\fp\in\Spec(\widehat{R})\backslash\{\m\}$ (see \cite[Theorem 3.1]{VelezOpennessFrational} or \cite[Theorem 7.8]{MaPolstraFsingularitiesBook}). Now we let $\widehat{R}\to \widehat{R}^\Gamma$ be the $\Gamma$-construction with respect to a sufficiently small $\Gamma$ such that
\begin{enumerate}
  \item $\widehat{R}^\Gamma$ is reduced and equidimensional.
  \item $\fq^*\widehat{R}^\Gamma=(\fq\widehat{R}^\Gamma)^*$ for all $\fq\subseteq R$ generated by a system of parameters.
  \item $0^*_{H_\m^d(\widehat{R}^\Gamma)}\cong 0^*_{H_\m^d(\widehat{R})}\otimes_{\widehat{R}}\widehat{R}^\Gamma=0^*_{H_\m^d(R)}\otimes_R\widehat{R}^\Gamma$.
\end{enumerate}
Note that such a choice of $\Gamma$ exists by Lemma \ref{lem: gamma construction reduced} and Lemma \ref{lem: gamma construction F-rat}.

We set $S:=\widehat{R}^\Gamma$ and $\n:=\m S$. Then $(R,\m, k)\to (S,\n, l)$ is a flat local extension such that $S/\m S=l$ is a field. It follows that
$$\tau^{<d}\mathbf{R}\Gamma_{\n}S\cong (\tau^{<d}\mathbf{R}\Gamma_\m R)\otimes_RS$$
is quasi-isomorphic to a complex of $l$-vector spaces as $\tau^{<d}\mathbf{R}\Gamma_\m R$ is quasi-isomorphic to a complex of $k$-vector spaces by Theorem \ref{thm: schenzel}. Thus we know that $S$ is a Buchsbaum ring by Theorem \ref{thm: schenzel} again.

By construction, $S$ is a reduced, equidimensional, F-finite Buchsbaum local ring such that $\fq^*S=(\fq S)^*$ for all $\fq\subseteq R$ generated by a system of parameters. We fix a test element $c\in S^\circ$ and fix $e>0$ such that $p^e$ is a test exponent for all $(c, \fq S)$, such a choice exists by Theorem \ref{thm: sharp text exponent}. It is easy to check that $F^e_*S$ is a (finitely generated) Buchsbaum $S$-module. Consider the following short exact sequence:
\begin{equation}
\label{eqn: short exact seq defining C}
0\to S \xrightarrow{\cdot F^e_*c} F^e_*S \to C\to 0.
\end{equation}
If we tensor the above short exact sequence with $R/\fq$, then we obtain:
\begin{equation}
\label{eqn: LES tight closure}
0\to (\fq S)^*/\fq S \to S/\fq S \xrightarrow{\cdot F^e_*c} F^e_*S/\fq \cdot F^e_*S \to C/\fq C\to 0.
\end{equation}
Here we crucially used the fact that $p^e$ is a test exponent of $(c,\fq)$ to identify the kernel of $S/\fq S \xrightarrow{\cdot F^e_*c} F^e_*S/\fq \cdot F^e_*S$ with $(\fq S)^*/\fq S$. Our key observation is the following.
\begin{claim}
\label{clm: key claim}
$C$ is a Buchsbaum $S$-module.
\end{claim}
\begin{proof}[Proof of Claim]
For every $\fq\subseteq R$ generated by a system of parameters, we have
$$e(\fq S, C)=e(\fq S, F^e_*S)- e(\fq S, S),$$
and by (\ref{eqn: LES tight closure}), we know that
$$\ell_S(C/\fq C) = \ell_S(F^e_*S/\fq\cdot F^e_*S) - \ell_S(S/\fq S) + \ell_S((\fq S)^*/\fq S).$$
Since $S$ and $F^e_*S$ are both Buchsbaum $S$-modules and
$$\ell_S((\fq S)^*/\fq S)=\ell_S(\fq^* S/\fq S)=\length_R(\fq^*/\fq)$$
is independent of the choice of $\fq$ by assumption. It follows that
\small
$$\ell_S(C/\fq C)-e(\fq S, C)=\left(\ell_S(F^e_*S/\fq\cdot F^e_*S) - e(\fq S, F^e_*S)\right) -\left(\ell_S(S/\fq S) -e(\fq S, S)\right) + \ell_S((\fq S)^*/\fq S) $$
\normalsize
is independent of $\fq$. This implies that every such $\fq$ is standard on $C$ (more precisely, every system of parameters of $R$ is a standard system of parameters of $C$). But since $\m S=\n$, we can fix a generating set of $\n$ that are elements in $\m$. By Theorem \ref{thm: finite criteion for Buchsbaum}, $C$ is a Buchsbaum $S$-module.
\end{proof}

At this point, if we compare (\ref{eqn: LES tight closure}) with the long exact sequence of Koszul homology with respect to $\fq=(x_1,\dots,x_d)$ induced by (\ref{eqn: short exact seq defining C}), we obtain a surjection
$$H_1(x_1,\dots,x_d; C)\twoheadrightarrow (\fq S)^*/\fq S \cong \fq^* S/\fq S.$$
It is well-known that, as $C$ is Buchsbaum, $H_1(x_1,\dots,x_d; C)$ is annihilated by $\n$ (for example see \cite[Theorem 2.3]{SchenzelApplicationsofDualizingComplex}). It follows that $\fq^* S/\fq S$ is annihilated by $\n$ and thus $\fq^*/\fq$ is annihilated by $\m$. This establishes $(2)\Rightarrow(3)$. We will not use this in the sequel though.

Now we return to the proof of $(2)\Rightarrow(4)$. Consider the exact triangle:
$$\mathbf{R}\Gamma_{\n} C[-1]\to \mathbf{R}\Gamma_{\n} S \xrightarrow{\cdot F^e_*c} \mathbf{R}\Gamma_{\n} F^e_*S\xrightarrow{+1} .$$
Since $F^e_*S$ is Buchsbaum, $H_{\n}^j(F^e_*S)$ is annihilated by $F^e_*\n$ for all $j<d$ (see Theorem \ref{thm: schenzel}), thus the long exact sequence of cohomology induced by the above triangle splits as
\begin{equation}
\label{eqn: split surj lower LC}
0\to H_{\n}^{j-1}(F^e_*S) \to h^j(\mathbf{R}\Gamma_{\n} C[-1])\to H_{\n}^j(S)\to 0
\end{equation}
for all $j<d$, and
$$
0\to H_{\n}^{d-1}(F^e_*S) \to h^d(\mathbf{R}\Gamma_{\n} C[-1])\to H_{\n}^d(S)\xrightarrow{\cdot F^e_*c} H_{\n}^d(F^e_*S).
$$
Since $(\fq S)^*/(\fq S)\cong \ker(S/\fq S \xrightarrow{\cdot F^e_*c} F^e_*S/\fq\cdot F^e_*S)$ all $\fq$ as in (\ref{eqn: LES tight closure}), by taking a direct limit for all $\fq^{[p^e]}$ we know that (see (\ref{eqn: top lc})):
$$\ker(H_{\n}^d(S)\xrightarrow{\cdot F^e_*c} H_{\n}^d(F^e_*S))=0^*_{H_{\n}^d(S)}.$$
Thus we have
\begin{equation}
\label{eqn: split surj top LC}
0\to H_{\n}^{d-1}(F^e_*S) \to h^d(\mathbf{R}\Gamma_{\n} C[-1])\to 0^*_{H_{\n}^d(S)}\to 0.
\end{equation}
Now by the definition of ${}^*$-truncation, we have the following diagram:
\[
\xymatrix{
\tau^{<d,*}\mathbf{R}\Gamma_{\n}S \ar[r] & \mathbf{R}\Gamma_{\n} S \cong \tau^{\leq d}\mathbf{R}\Gamma_{\n} S \ar[r] & H_{\n}^d(S)/0^*_{H_{\n}^d(S)}[-d] \ar[r]^-{+1} &  \\
& \tau^{\leq d} (\mathbf{R}\Gamma_{\n} C[-1]) \ar[u] \ar@{.>}[ur]^-0&
}
\]
where the dotted arrow is the zero map in $D(S)$: this is because the image of $h^d(\mathbf{R}\Gamma_{\n} C[-1])$ is zero in $H_{\n}^d(S)/0^*_{H_{\n}^d(S)}$, but $\tau^{\leq d} (\mathbf{R}\Gamma_{\n} C[-1])$ lives in cohomology degree $\leq d$ while $H_{\n}^d(S)/0^*_{H_{\n}^d(S)}[-d]$ only lives in cohomology degree $d$. It follows that there is an induced map in $D(S)$:
$$\tau^{\leq d} (\mathbf{R}\Gamma_{\n} C[-1]) \to \tau^{<d,*}\mathbf{R}\Gamma_{\n}S.$$
Moreover, the induced maps on cohomology
$$h^j(\tau^{\leq d} (\mathbf{R}\Gamma_{\n} C[-1]))\cong h^j(\mathbf{R}\Gamma_{\n} C[-1])\twoheadrightarrow H_{\n}^j(S), \text{ where $j<d$, and }$$
$$h^d(\tau^{\leq d} (\mathbf{R}\Gamma_{\n} C[-1]))\cong h^d(\mathbf{R}\Gamma_{\n} C[-1])\twoheadrightarrow 0^*_{H_{\n}^d(S)}$$
are all surjective by (\ref{eqn: split surj lower LC}) and (\ref{eqn: split surj top LC}).

We now invoke Theorem \ref{thm: surjectivity criterion for Buchsbaum}. Since $C$ is a Buchsbaum $S$-module by Claim \ref{clm: key claim}, Theorem \ref{thm: surjectivity criterion for Buchsbaum} tells us that
$$H^j({\n}, C)\twoheadrightarrow H_{\n}^j(C)$$ is surjective for all $j\leq d-1$. That is, if we consider the natural map
$$K^\bullet({\n}, C)[-1]\to \mathbf{R}\Gamma_{\n} C[-1],$$
where $K^\bullet({\n}, C)$ denote the cohomological Koszul complex on a generating set of $\n$, then the induced map on cohomology
$$h^j(K^\bullet({\n}, C)[-1]) \to h^j(\mathbf{R}\Gamma_{\n} C[-1])$$
is surjective for all $j\leq d$.

Putting all these together, if we consider the induced composition map
$$\tau^{\leq d}(K^\bullet({\n}, C)[-1]) \to \tau^{\leq d} (\mathbf{R}\Gamma_{\n} C[-1]) \to \tau^{<d,*}\mathbf{R}\Gamma_{\n}S,$$
then the induced map on each cohomology is surjective. But $K^\bullet({\n}, C)$ is quasi-isomorphic to a complex of $l$-vector spaces: for example one can choose a regular local ring $(A,\n, l)$ that surjects onto $S$ and note that $K^\bullet({\n}, C)\cong \mathbf{R}\Hom_A(l, C)$, the latter is obviously a complex of $l$-vector spaces. Now we have $\tau^{\leq d}(K^\bullet({\n}, C)[-1])$ is quasi-isomorphic to a complex of $l$-vector spaces, and it follows from \cite[II, Proposition 4.3]{StuckradVogelBuchsbaumRingsBook} (or use the dual statement of \cite[Lemma 2.4]{BhattMaSchwedeDualizingComplexFinjectiveDB}) that $\tau^{<d,*}\mathbf{R}\Gamma_{\n}S$ is quasi-isomorphic to a complex of $l$-vector spaces. In particular, it is quasi-isomorphic to a complex of $k$-vector spaces.

Finally, since $0^*_{H_{\n}^d(S)}\cong 0^*_{H_\m^d(R)}\otimes_RS$, it is easy to see that
$$\tau^{<d, *}\mathbf{R}\Gamma_{\n}S\cong (\tau^{<d, *}\mathbf{R}\Gamma_\m R)\otimes_RS.$$
Since $(R,\m, k)\to (S, \n, l)$ is faithfully flat, the natural map
$$\tau^{<d, *}\mathbf{R}\Gamma_\m R\to (\tau^{<d, *}\mathbf{R}\Gamma_\m R)\otimes_RS \cong \tau^{<d, *}\mathbf{R}\Gamma_{\n}S$$
induces an injection on each cohomology. As $\tau^{<d, *}\mathbf{R}\Gamma_{\n}S$ is quasi-isomorphic to a complex of $k$-vector spaces, it follows from \cite[Lemma 2.4]{BhattMaSchwedeDualizingComplexFinjectiveDB} that $\tau^{<d, *}\mathbf{R}\Gamma_\m R$ is quasi-isomorphic to a complex of $k$-vector spaces. This completes the proof of $(2)\Rightarrow(4)$.
\end{proof}

\begin{proof}[Proof of $(4)\Rightarrow(3)$]First of all, since $\tau^{<d,*}\mathbf{R}\Gamma_\m R$ is quasi-isomorphic to a complex of $k$-vector spaces, $H_\m^i(R)$ and $0^*_{H_\m^d(R)}$ are $k$-vector spaces for all $i<d$ (in particular, they have finite length). Thus by Theorem \ref{thm: GotoNakamura}, we know that $R_{\fp}$ is F-rational (and hence regular) for all minimal primes $\fp$ of $R$. Since $R$ is unmixed, it follows that $R$ is reduced and equidimensional.

Now let $R^+$ be the absolute integral closure of $R$ and let $\underline{x}:=x_1,\dots,x_d$ be any system of parameters. The short exact sequence
$$0\to R\to R^+\to R^+/R\to 0$$
induces a diagram (where $K_\bullet(\underline{x}, R)$ denotes the usual homological Koszul complex):
\[\xymatrix{
K_\bullet(\underline{x}, R) \otimes \mathbf{R}\Gamma_\m(R^+/R)[-1] \ar[r] \ar[d]^\cong & K_\bullet(\underline{x}, R) \otimes \mathbf{R}\Gamma_\m R  \ar[r] \ar[d]^\cong &  K_\bullet(\underline{x}, R) \otimes \mathbf{R}\Gamma_\m R^+ \ar[d]^\cong  \ar[r]^-{+1} & \\
K_\bullet(\underline{x}, R^+/R)[-1] \ar[r] &   K_\bullet(\underline{x}, R)  \ar[r] &  K_\bullet(\underline{x}, R^+) \ar[r]^-{+1} \ar[r]^-{+1} &
}.
\]
Taking the cohomology of the second line, we obtain (set $\fq=(x_1,\dots,x_d)$):
$$0=h^{-1}K_\bullet(\underline{x}, R^+)\to h^0(K_\bullet(\underline{x}, R^+/R)[-1]) \to R/\fq \to R^+/\fq R^+$$
where the leftmost $0$ follows from the fact that $\underline{x}=x_1,\dots,x_d$ is a regular sequence on $R^+$ by Theorem \ref{thm: HochsterHunekebigCM}. It follows that
\begin{equation}
\label{eqn: h^0 is tight closure}
h^0(K_\bullet(\underline{x}, R) \otimes \mathbf{R}\Gamma_\m(R^+/R)[-1])\cong h^0(K_\bullet(\underline{x}, R^+/R)[-1])\cong \fq^+/\fq \cong \fq^*/\fq
\end{equation}
where the last isomorphism follows from Theorem \ref{thm: SmithPlusClosure}.

Next, we note that the exact triangle
$$\tau^{\leq d}(\mathbf{R}\Gamma_\m(R^+/R)[-1]) \to \mathbf{R}\Gamma_\m(R^+/R)[-1] \to \tau^{>d}(\mathbf{R}\Gamma_\m(R^+/R)[-1]) \xrightarrow{+1}$$
induces an exact triangle
\begin{eqnarray}
\label{eqn: exact triangle}
K_\bullet(\underline{x}, R) \otimes \tau^{\leq d}(\mathbf{R}\Gamma_\m(R^+/R)[-1]) &\to& K_\bullet(\underline{x}, R) \otimes \mathbf{R}\Gamma_\m(R^+/R)[-1] \\
&\to& K_\bullet(\underline{x}, R) \otimes \tau^{>d}(\mathbf{R}\Gamma_\m(R^+/R)[-1]) \xrightarrow{+1}. \notag
\end{eqnarray}
Since $\tau^{>d}(\mathbf{R}\Gamma_\m(R^+/R)[-1])$ lives in cohomology degree $>d$ and $K_\bullet(\underline{x}, R)$ is a complex of finite free $R$-modules whose terms sit in cohomology degree $[-d, 0]$, we have
$$h^0\left(K_\bullet(\underline{x}, R) \otimes \tau^{>d}(\mathbf{R}\Gamma_\m(R^+/R)[-1])\right)=h^{-1}\left(K_\bullet(\underline{x}, R) \otimes \tau^{>d}(\mathbf{R}\Gamma_\m(R^+/R)[-1])\right)=0.$$
Thus taking the cohomology of (\ref{eqn: exact triangle}), we obtain that
$$h^0\left(K_\bullet(\underline{x}, R) \otimes \tau^{\leq d}(\mathbf{R}\Gamma_\m(R^+/R)[-1])\right)\cong h^0\left(K_\bullet(\underline{x}, R) \otimes \mathbf{R}\Gamma_\m(R^+/R)[-1]\right).$$
Combining this with (\ref{eqn: h^0 is tight closure}) we have
$$\fq^*/\fq\cong h^0\left(K_\bullet(\underline{x}, R) \otimes \tau^{\leq d}(\mathbf{R}\Gamma_\m(R^+/R)[-1])\right).$$
Finally, by Lemma \ref{lem: alternative description}, $\tau^{\leq d}(\mathbf{R}\Gamma_\m(R^+/R)[-1])\cong \tau^{<d,*}\mathbf{R}\Gamma_\m R$ is quasi-isomorphic to a complex of $k$-vector spaces. It follows that
$$h^0\left(K_\bullet(\underline{x}, R) \otimes \tau^{\leq d}(\mathbf{R}\Gamma_\m(R^+/R)[-1])\right)$$ is a $k$-vector space. Therefore $\fq^*/\fq$ is annihilated by $\m$ for every $\fq$ generated by a system of parameters, i.e., $\tau_{\pa}(R)$ contains $\m$.
\end{proof}
\end{proof}

\begin{remark}
\label{rmk: tau radical F-inj}
Suppose $(R,\m)$ is an equidimensional excellent local ring such that $R_{\fp}$ is F-rational for all $\fp\in\Spec(R)\backslash\{\m\}$. It is well-known that under these assumptions, the parameter test ideal $\tau_{\pa}(R)$ is $\m$-primary (for example, see \cite[Theorem 6.8]{KatzmanMurruVelezZhangGlobalParmeterTestIdeal}). If, additionally, $R$ is F-injective, then we claim that $\tau_{\pa}(R)$ contains $\m$. It is enough to show that $\tau_{\pa}(R)$ is a radical ideal, but if $r^{p^e}\in \tau_{\pa}(R)$ then for every ideal $\fq$ generated by a system of parameters, we have $$(r\fq^*)^{[p^e]}=r^{p^e}(\fq^*)^{[p^e]}\subseteq r^{p^e}(\fq^{[p^e]})^*\subseteq \fq^{[p^e]},$$
and it follows from \cite[Theorem 1.1]{MaFinjectiveBuchsbaum} that $r\fq^*\subseteq \fq$ and hence $r\in \tau_{\pa}(R)$ as wanted. Thus, Theorem \ref{thm: main theorem} $(3)\Rightarrow(4)$ should be viewed as a generalization of Theorem \ref{thm: tight closure truncation}, and Theorem \ref{thm: main theorem} $(3)\Rightarrow(2)$ should be viewed as a generalization of \cite[Main Theorem]{PhamHungQuyTightClosureParameterIdeal}.
\end{remark}

Lastly, we point out that even in dimension one, the unmixed assumption is necessary in Theorem \ref{thm: main theorem} (at least for $(2)\Rightarrow(3)$), as the following example shows.

\begin{example}
Let $k$ be a field of characteristic $p>0$ and let $(R,\m)=k[[x,y]]/(xy) \cap (x,y)^3$.
Then $R$ is a one-dimensional excellent local ring. Note that $R$ is not unmixed: $\m=(x,y)$ is an associated prime of $R$. We claim that
$\length(\fq^*/\fq)=2$ for every $\fq=(z)\nsubseteq (x)\cup (y)$ but $\tau_{\pa}(R)\nsupseteq\m$. To see this, first note that $$H_\m^0(R)=k\cdot \langle xy \rangle=(xy)\subseteq R$$ and that $$\overline{R}:= R/H_\m^0(R)= R/(xy)\cong k[[x,y]]/(xy). $$
It follows that $R$ is Buchsbaum (as $H_\m^0(R)$ is annihilated by $\m$), and since $H_\m^1(R)\cong H_\m^1(\overline{R})$, we have $$0^*_{H_\m^1(R)}\cong 0^*_{H_\m^1(\overline{R})}= k\cdot \langle \frac{1}{x+y}\rangle.$$
Now for each $\fq=(z)\nsubseteq (x)\cup (y)$, since $R$ is Buchsbaum, we know by Lemma \ref{lem: length of q^lim/q} that $$\length(\fq^{\lim}/\fq)=\length(H_\m^0(R))=1.$$
By Remark \ref{rmk: lim closure tight closure} we know that $\fq^*/\fq^{\lim}\hookrightarrow 0^*_{H_\m^1(R)}$. But $\fq^{\lim}\overline{R}=\fq\overline{R}$ since $\overline{R}$ is Cohen-Macaulay while $\fq^*\overline{R}=(\fq\overline{R})^*\neq \fq\overline{R}$ as $\overline{R}$ is not F-rational (where the first equality follows from \cite[Proposition 4.1 (j)]{HochsterHuneke1}). It follows that $\fq^*\neq \fq^{\lim}$ and thus $\fq^*/\fq^{\lim}\cong 0^*_{H_\m^1(R)}$ since $\length(0^*_{H_\m^1(R)})=1$. Therefore we have
$$\length(\fq^*/\fq)=\length(\fq^*/\fq^{\lim})+\length(\fq^{\lim}/\fq)=2.$$
Finally, note that $x+y$ is a minimal reduction of $\m=(x,y)$, and since tight closure agrees with integral closure for principal ideals (see \cite[Corollary 5.8]{HochsterHuneke1}), we have $(x+y)^*=\overline{(x+y)}=\m$. Since $\m^2\nsubseteq (x+y)$ by a simple computation, we have $\tau_{\pa}(R)\nsupseteq\m$.
\end{example}

\bibliographystyle{plain}
\bibliography{refs}

\end{document}